\documentclass[draft,11pt,english]{article}
\usepackage{amsfonts,amsmath,amsthm,amscd,amssymb,latexsym}

    \usepackage[T2A]{fontenc}
    \usepackage[cp1251]{inputenc}
    \usepackage[russian]{babel}

\begin{document}

\title{О НЕРАВЕНСТВЕ ТИПА ВЯЙСЯЛЯ ДЛЯ УГЛОВОЙ ДИЛАТАЦИИ ОТОБРАЖЕНИЙ И НЕКОТОРЫХ ЕГО ПРИЛОЖЕНИЯХ}

\author{Евгений А. Севостьянов, Руслан Р. Салимов}

\date{16.03.2014}

\theoremstyle{plain}
\newtheorem{theorem}{Теорема}[section]
\newtheorem{lemma}{Лемма}[section]
\newtheorem{proposition}{Предложение}[section]
\newtheorem{corollary}{Следствие}[section]
\newtheorem{definition}{Определение}[section]
\theoremstyle{definition}
\newtheorem{example}{Пример}[section]
\newtheorem{remark}{Замечание}[section]
\newcommand{\keywords}{\textbf{Key words.  }\medskip}
\newcommand{\subjclass}{\textbf{MSC 2000. }\medskip}
\renewcommand{\abstract}{\textbf{Аннотация.  }\medskip}
\numberwithin{equation}{section}

\maketitle

\medskip


\begin{abstract}
Для одного подвида отображений с конечным искажением $f:D\rightarrow
D^{\prime},$ $D, D^{\prime}\subset{\Bbb R}^n,$ $n\ge 2,$ допускающих
наличие точек ветвления, установлено некоторое модульное
неравенство, играющее существенную роль при исследовании различных
проблем плоских и пространственных отображений. В качестве одного из
приложений полученных результатов исследован вопрос об устранении
изолированной особенности открытых дискретных отображений с конечным
искажением длины.
\end{abstract}

\section{Введение}
Настоящая статья посвящена изучению отображений с конечным
искажением, активно изучаемых в последнее время (см., напр.,
\cite{AC$_1$}, \cite{BGMV}, \cite{Cr$_1$}, \cite{GG}, \cite{GRSY},
\cite{IM}, \cite{MRSY} и \cite{MRSY$_1$}). Речь идёт, прежде всего,
об отображениях с конечным искажением длины, являющихся некоторым
подвидом отображений с конечным искажением и содержащим в себе класс
отображений с ограниченным искажением по Решетняку (см.
\cite{Re$_1$} и \cite{Ri}). В данной работе устанавливается
некоторый аналог результатов публикации \cite{Va$_1$}, где для
отображений с ограниченным искажением вначале было доказано ключевое
модульное неравенство, а затем на его основе решён вопрос об
устранении изолированной особенности отображений. Здесь та же
проблема решена для отображений с конечным искажением длины, при
этом, накладываемые условия относятся к специально подобранному
дилатационному коэффициенту, имеющему локальный характер (что, в
частности, обобщает результаты, опубликованные одним из авторов
ранее, см. \cite{Sev$_1$} и \cite{Sev$_2$}). Отметим, что
отображения с конечным искажением длины введены О. Мартио совместно
с В. Рязановым, У. Сребро и Э. Якубовым (\cite{MRSY$_1$}) и
представляют собой одно из обобщений отображений с ограниченным
искажением по Решетняку (см. \cite{Re$_1$} и \cite{Ri}). Отображения
с конечным искажением длины могут быть определены как отображения,
искажающие евклидово расстояние в конечное число раз в почти всех
точках, а также обладающие $N$-свойством Лузина относительно меры
Лебега в ${\Bbb R}^n$ и меры длины на кривых в прямую и обратную
стороны (см. там же).

\medskip
Опишем вкратце цель исследований настоящей статьи. Всюду далее $D$
-- область в ${\Bbb R}^n,$ $n\ge 2,$ $m$ -- мера Лебега ${\Bbb
R}^n,$ ${\rm dist\,}(A,B)$ -- ев\-кли\-дово расстояние между
множествами $A, B\subset {\Bbb R}^n,$ $(x,y)$ обозначает
(стандартное) скалярное произведение векторов $x,y\in {\Bbb R}^n,$
$B(x_0, r)=\left\{x\in{\Bbb R}^n: |x-x_0|< r\right\},$ ${\Bbb B}^n
:= B(0, 1),$ $S(x_0,r) = \{ x\,\in\,{\Bbb R}^n : |x-x_0|=r\},$
${\Bbb S}^{n-1}:=S(0, 1),$ $\omega_{n-1}$ означает площадь сферы
${\Bbb S}^{n-1}$ в ${\Bbb R}^n,$ $\Omega_{n}$ -- объём единичного
шара ${\Bbb B}^{n}$ в ${\Bbb R}^n,$ $f^{\,\prime}(x)$ обозначает
матрицу Якоби отображения $f$ в точки $x\in D$ её
дифференцируемости, запись $f:D\rightarrow {\Bbb R}^n$ предполагает,
что отображение $f,$ заданное в области $D,$ непрерывно. Отображение
$f:D\rightarrow {\Bbb R}^n$ называется {\it дискретным}, если
прообраз $f^{-1}\left(y\right)$ каждой точки $y\,\in\,{\Bbb R}^n$
состоит из изолированных точек, и {\it открытым}, если образ любого
открытого множества $U\subset D$ является открытым множеством в
${\Bbb R}^n.$ Будем говорить, что отображение $f:D\rightarrow {\Bbb
R}^n$ обладает {\it $N$--свой\-ст\-вом Лузина,} или просто
$N$--свой\-ст\-вом, если из условия $m(E)=0,$ $E\subset D,$ следует,
что $m(f(E))=0.$ Аналогично, говорят, что отображение
$f:D\rightarrow {\Bbb R}^n$ обладает {\it $N^{-1}$--свой\-ст\-вом,}
если из условия $m(E)=0,$ $E\subset {\Bbb R}^n,$ следует, что
$m\left(f^{-1}(E)\right)=0,$ где, как обычно, запись $f^{-1}(E)$
обозначает полный прообраз множества $E$ при отображении $f.$ Везде
ниже мы подразумеваем, что отображение $f$ сохраняет ориентацию,
т.е., топологическая степень отображения $\mu (y, f, G)>0$ для всех
указанных выше $y$ и $G,$ если не оговорено противное. Пусть
$f:D\rightarrow {\Bbb R}^n$ -- произвольное отображение и пусть
существует область $G\subset D,$ $\overline{G}\subset D,$ такая, что
$\overline{G}\cap f^{\,-1}\left(f(x)\right)=\left\{x\right\}.$ Тогда
величина $\mu\, \left(f(x), f, G\right),$ называемая {\it локальным
топологическим индексом},  не зависит от выбора области $G$ и
обозначается символом $i(x,f).$ Напомним, что точка $x_0\in D$
называется {\it точкой ветвления отображения $f$}, если ни в какой
окрестности $U$ точки $x_0$ сужение $f|_U$ не является
гомеоморфизмом. Множество точек ветвления отображения $f$ принято
обозначать символом $B_f.$ Очевидно, если $x_0\in D\setminus B_f$ и
$f$ -- сохраняющее ориентацию открытое дискретное отображение, то
$i(x_0, f)=1.$

\medskip Здесь и далее {\it кривой} $\gamma$ мы называем непрерывное
отображение отрезка $[a,b]$ (открытого интервала $(a,b),$ либо
полуоткрытого интервала $[a,b)$ или $(a,b]$) в ${\Bbb R}^n,$
$\gamma:[a,b]\rightarrow {\Bbb R}^n.$ Под семейством кривых $\Gamma$
подразумевается некоторый фиксированный набор кривых $\gamma,$ а
$f(\Gamma)=\left\{f\circ\gamma|\gamma\in\Gamma\right\}.$ Борелева
функция $\rho:{\Bbb R}^n\,\rightarrow [0,\infty] $ называется {\it
допустимой} для семейства $\Gamma$ кривых $\gamma$ в ${\Bbb R}^n,$
если соотношение
\begin{equation}\label{eq31*}
\int\limits_{\gamma}\rho (x)\ \ |dx|\ge 1
\end{equation}
выполнено для всех кривых $ \gamma \in \Gamma.$ В этом случае мы
пишем: $\rho \in {\rm adm} \,\Gamma.$ {\it Модулем} семейства кривых
$\Gamma $ называется величина
$$M(\Gamma)=\inf_{\rho \in \,{\rm adm}\,\Gamma}
\int\limits_D \rho ^n (x)\ \ dm(x)\,.$$ Свойства модуля в некоторой
мере аналогичны свойствам меры Лебега $m$ в ${\Bbb R}^n.$ Именно,
модуль пустого семейства кривых равен нулю, $M(\varnothing)=0,$
модуль обладает свойством монотонности относительно семейств кривых
$\Gamma_1$ и $\Gamma_2:$ $\Gamma_1\subset\Gamma_2\Rightarrow
M(\Gamma_1)\le M(\Gamma_2),$ а также свойством полуаддитивности:
$M\left(\bigcup\limits_{i=1}^{\infty}\Gamma_i\right)\le
\sum\limits_{i=1}^{\infty}M(\Gamma_i)$ (см. \cite[теорема~6.2]{Va}).
Говорят, что семейство кривых $\Gamma_1$ \index{минорирование}{\it
минорируется} семейством $\Gamma_2,$ пишем $\Gamma_1\,>\,\Gamma_2,$
если для каждой кривой $\gamma\,\in\,\Gamma_1$ существует подкривая,
которая принадлежит семейству $\Gamma_2.$
В этом случае,
\begin{equation}\label{eq32*A}
\Gamma_1
> \Gamma_2 \quad \Rightarrow \quad M(\Gamma_1)\le M(\Gamma_2)
\end{equation} (см. \cite[теорема~6.4, гл.~I]{Va}).
Говорят, что некоторое свойство выполнено для {\it почти всех (п.в.)
кривых} области $D$, если оно имеет место для всех кривых, лежащих в
$D$, кроме некоторого их семейства, модуль которого равен нулю.
Пусть $\Delta \subset \Bbb R$ -- открытый интервал числовой прямой,
$\gamma: \Delta\rightarrow {\Bbb R}^n$ -- локально спрямляемая
кривая. В таком случае, очевидно, существует единственная
неубывающая функция длины $l_{\gamma}:\Delta\rightarrow
\Delta_{\gamma}\subset \Bbb{R}$ с условием $l_{\gamma}(t_0)=0,$ $t_0
\in \Delta,$ такая что значение $l_{\gamma}(t)$ равно длине
подкривой $\gamma\mid_{[t_0, t]}$ кривой $\gamma,$ если $t>t_0,$ и
$-l\left(\gamma\mid_ {[t,\,t_0]}\right),$ если $t<t_0,$ $t\in
\Delta.$ Пусть $g:|\gamma|\rightarrow {\Bbb R}^n$ -- непрерывное
отображение, где $|\gamma| = \gamma(\Delta)\subset \Bbb{R}^n.$
Предположим, что кривая $\widetilde{\gamma}=g\circ \gamma$ также
локально спрямляема. Тогда, очевидно, существует единственная
неубывающая функция $L_{\gamma,\,g}:\,\Delta_{\gamma} \rightarrow
\Delta_{\widetilde{\gamma}}$ такая, что
$L_{\gamma,\,g}\left(l_{\gamma}\left(t\right)
 \right)\,=\,l_{\widetilde{\gamma}}\left(t\right)$ при всех
$t\in\Delta.$ Кривая $\gamma \in D$ называется {\it (полным)
поднятием кривой $\widetilde{\gamma}\in {\Bbb R}^n$ при отображении
$f:D \rightarrow {\Bbb R}^n,$} если $\widetilde{\gamma}=f \circ
\gamma.$

Говорят, что отображение $f:D\rightarrow {\Bbb R}^n$ принадлежит
классу $ACP$ в области $D,$ пишем $f\in ACP,$ если, для почти всех
кривых $\gamma$ в области $D,$ кривая
$\widetilde{\gamma}=f\circ\gamma$ локально спрямляема и функция
длины $L_{\gamma,\,f},$ введённая выше, абсолютно непрерывна на всех
замкнутых интервалах, лежащих в $\Delta_{\gamma},$ для почти всех
кривых $\gamma$ в $D.$ Предположим, что $f:D\rightarrow {\Bbb R}^n$
-- дискретное отображение, тогда может быть определена функция
$L^{-1}_{\gamma,\,f}.$ В таком случае, будем говорить, что $f$
обладает {\it свойством $ACP^{\,-1}$} в области $D,$ пишем $f\in
ACP^{-1},$ если для почти всех кривых $\widetilde{\gamma}\in f(D)$
каждое поднятие $\gamma$ при отображении $f,$
$f\circ\gamma=\widetilde{\gamma},$ является локально спрямляемой
кривой и, кроме того, обратная функция $L^{-1}_{\gamma,\,f}$
абсолютно непрерывна на всех замкнутых интервалах, лежащих в
$\Delta_{\widetilde{\gamma}},$ для почти всех кривых
$\widetilde{\gamma}$ в $f(D)$ и каждого поднятия $\gamma$ кривой
$\widetilde{\gamma}=f\circ\gamma.$

Пусть $f:D\rightarrow {\Bbb R}^n$ дискретное отображение, тогда $f$
будем называть {\it отображением с конечным искажением длины,} пишем
$f\in FLD,$ если $f$ дифференцируемо почти всюду в $D,$ $f$ обладает
$N$ и $N^{-1}$--свой\-ст\-ва\-ми, и, кроме того, $f\in ACP\cap
ACP^{-1}.$

\medskip
Здесь и далее {\it внутренняя дилатация} $K_I(x,f)$ отображения $f$
в точке $x$ определяется при $J(x, f)\ne 0$ отношением
\begin{equation}\label{equa16}
K_I(x,f)\quad=\quad\frac{|J(x,f)|}{{l\left(f^{\,\prime}(x)\right)}^n}\,,
\end{equation}
где $J(x,f)={\rm det\,}f^{\,\prime}(x)$ обозначает якобиан
отображения $f$ в точке $x,$ а
$l\left(f^{\,\prime}(x)\right):=\min\limits_{|h|=1}
|f^{\,\prime}(x)h|.$ Полагаем $K_I(x,f)=1,$ если
$f^{\,\prime}(x)=0,$ и $K_I(x,f)=\infty.$

\medskip В дальнейшем для кривых $\alpha:J\rightarrow {\Bbb R}^n$ и $\beta:I\rightarrow {\Bbb
R}^n$ и отрезков $J\subset I\subset {\Bbb R}$  запись $\alpha\subset
\beta$ означает, что $\beta|_{J}=\alpha,$ т.е., что кривая $\alpha$
является подкривой кривой $\beta.$ Пусть $f:D\rightarrow {\Bbb R}^n$
-- отображение с ограниченным искажением (см. \cite{Re$_1$} и
\cite{Ri}), $\Gamma$ -- семейство кривых в $D,$ $\Gamma^{\,\prime}$
-- семейство кривых в ${\Bbb R}^n$ и $m$ -- натуральное число, такое
что выполнено следующее условие. Для каждой кривой $\beta\in
\Gamma^{\,\prime}$ найдутся кривые $\alpha_1,\ldots,\alpha_m$
семейства $\Gamma$ такие что $f\circ \alpha_j\subset \beta$ для всех
$j$ и равенство $\alpha_j(t)=x$ имеет место при всех $x\in D,$ всех
$t$ и не более чем $i(x,f)$ индексах $j$ (где $i(x, f)$ -- локальный
топологический индекс отображения $f$ в точке $x$). Тогда
\begin{equation}\label{equa8}
M(\Gamma^{\,\prime} )\quad\le\quad\frac{1}{m}\,\cdot \,M(\Gamma)
\end{equation}
(см. \cite[теорема~3.1]{Va$_1$} либо \cite[разд.~9, гл.~II]{Ri}).
Неравенство (\ref{equa8}) установлено Ю.~Вяйсяля в \cite{Va$_1$} и
играет существенную роль в исследовании проблемы об изолированной
особенности, а также теории распределения значений (см.
\cite{Va$_1$} и \cite{Ri}). Одним из авторов данной работы был
установлен некоторый аналог неравенства типа (\ref{equa8}) для
отображений с конечным искажением длины (см. \cite{Sev$_1$}), а
именно, было показано, что для открытых дискретных отображений
$f:D\rightarrow {\Bbb R}^n$ с конечным искажением длины вместо
неравенства (\ref{equa8}) имеет место соотношение
\begin{equation}\label{equa1}
M(\Gamma^{\,\prime} )\quad\le\quad \frac{1}{m}\quad\int\limits_D
K_I(x,\,f)\cdot \rho^n(x)\, dm(x)\,,\end{equation}
%
%
выполненное для любого семейства $\Gamma $ путей $\gamma$ в $D$ и
для каждой $\rho \in {\rm }\,{\rm adm}\,\Gamma.$

\medskip
Отметим, что в ряде важных для приложений случаев неравенство
(\ref{equa1}) может оказаться несколько грубым, поскольку для
исследования отдельных вопросов возможно использование лишь
локальных ограничений в окрестности заданной точки, а не оценок,
задействующих всевозможные семейства кривых (и являющихся, таким
образом, ограничениями глобального характера). Соответствующие
примеры, построенные в заключительной части работы, указывают на
упомянутые отличия.

Настоящая работа, прежде всего, посвящена установлению некоторых
локальных оценок, аналогичных (\ref{equa1}). Уточнение
(\ref{equa1}), которое будет здесь проделано, касается двух
отдельных $"$направлений$"$: с одной стороны, вместо произвольных
семейств кривых в (\ref{equa1}) будут рассматриваться лишь кривые,
соединяющие обкладки сферического кольца с центром в данной точке; с
другой стороны, вместо величины $K_I(x, f)$ будет участвовать другая
величина, не превосходящая $K_I(x, f)$ и также имеющая локальный
характер (зависит от конкретной точки $x_0,$ в окрестности которой и
рассматривается основное неравенство (\ref{equa1})).

\medskip
Для этой цели рассмотрим также следующие определения и обозначения.
Пусть $x_0\in {\Bbb R}^n,$ тогда для отображения $f:D\rightarrow
{\Bbb R}^n$ полагаем в точке дифференцируемости $x\in D$
$$l_f(x, x_0)=\min\limits_{|h|=1}\frac{|\partial_h f(x)|}{\left|\left(h, \frac{x-x_0}{|x-x_0|}\right)\right|}\,,$$
где $\partial_h f(x)=\lim\limits_{t\rightarrow
+0}\frac{f(x+th)-f(x)}{t}$ -- производная по направлению $h$
отображения $f$ в точке $x.$ Введём в рассмотрение следующую
величину, называемую {\it угловой дилатацией в точке $x$ по
отношению к точке $x_0\in D,$} полагая в точке $x$
дифференцируемости и невырожденности отображения $f$
\begin{equation}\label{equa11}
D_f(x, x_0)=\frac{|J(x ,f)|}{l_f^n(x, x_0)}\,.
\end{equation}
При этом, величину $D_f(x, x_0)$ полагаем равной единице в точках
$x,$ где $f^{\,\prime}(x)=0$ и бесконечности, если $J(x, f)=0,$ но,
в то же время, $J(x, f)\ne 0.$ Отметим, что
$
D_f(x, x_0)\le K_I(x, f)
$
во всех точках $x,$ так как $\frac{1}{l_f(x, x_0)}\le
\frac{1}{l(f^{\,\prime}(x))}.$

Обозначим $A(r_1, r_2, x_0)=\{x\in {\Bbb R}^n: r_1<|x-x_0|<r_2\}$
при произвольных $0<r_1<r_2<\infty.$ Далее символ $\Gamma(E,F,D)$
означает семейство всех кривых $\gamma:[a,b]\rightarrow{\Bbb R}^n,$
которые соединяют $E$ и $F$ в $D,$ т.е. $\gamma(a)\in E,$
$\gamma(b)\in F$ и $\gamma(t)\in D$ при $t\in (a, b).$ Одним из
наиболее важных результатов настоящей работы является следующее
утверждение, установленное в частном случае $m=1$ для гомеоморфизмов
класса Соболева $f\in W_{loc}^{1, 2},$ $f^{\,-1}\in W_{loc}^{1, 2}$
при $n=2$ в работе \cite[теорема~2.17]{RSY} и для квазиконформных
отображений при $n\ge 2$ в работе \cite[Лемма~2.4]{GG}.

\begin{theorem}\label{th3.1}{\sl
Пусть $f:D\rightarrow {\Bbb R}^n$ -- открытое дискретное
дифференцируемое почти всюду отображение, обладающее $N,$ $N^{\,-1}$
и $ACP^{\,-1}$-свойствами, $x_0\in D,$ $A(r_1, r_2, x_0)\subset D,$
$\Gamma^{\,\prime}$ -- некоторое семейство кривых в ${\Bbb R}^n$ и
$m$ -- натуральное число, такое что выполнено следующее условие. Для
каждой кривой $\beta\in \Gamma^{\,\prime}$ найдутся кривые
$\alpha_1,\ldots,\alpha_m$ семейства $\Gamma(S(x_0, r_1), S(x_0,
r_2), A(r_1, r_2, x_0))$ такие что $f\circ \alpha_j\subset \beta$
для всех $j$ и равенство $\alpha_j(t)=x$ имеет место при всех $x\in
D,$ всех $t$ и не более чем $i(x,f)$ индексах $j.$ Тогда
\begin{equation}\label{eq5}
M(\Gamma^{\,\prime} )\quad\le\quad\frac{1}{m}\quad\int\limits_D
D_f(x, x_0)\cdot \rho^n (x)\, dm(x)
\end{equation}
для каждой неотрицательной измеримой по Лебегу функции $\rho:[r_1,
r_2]\rightarrow {\Bbb R}$ такой, что
$\int\limits_{r_1}^{r_2}\rho(r)dr\ge 1.$}
\end{theorem}

\begin{corollary}\label{cor1}
{\sl Заключение теоремы \ref{th3.1} выполнено, если отображение $f$
в условиях этой теоремы является отображением с конечным искажением
длины.}
\end{corollary}

\medskip
Всюду далее $q_{x_0}(r)$ означает среднее интегральное значение
$Q(x)$ над сферой $S(x_0, r),$
\begin{equation}\label{eq17}
q_{x_0}(r):=\frac{1}{\omega_{n-1}r^{n-1}}\int\limits_{S(x_0, r)}
Q(x)\,d{\cal H}^{n-1}\,,
\end{equation}
где ${\cal H}^{n-1}$ -- $(n-1)$-мерная мера Хаусдорфа.
Будем говорить, что функция ${\varphi}:D\rightarrow{\Bbb R}$ имеет
{\it конечное среднее колебание} в точке $x_0\in D$, пишем
$\varphi\in FMO(x_0),$ если
${\limsup\limits_{\varepsilon\rightarrow 0}}\
\frac{1}{\Omega_n\cdot\varepsilon^n} \int\limits_{B( x_0,
\varepsilon)} |{\varphi}(x)-\overline{{\varphi}}_{\varepsilon}|\
dm(x)<\infty,$
где $\overline{{\varphi}}_{\varepsilon}=
\frac{1}{\Omega_n\cdot\varepsilon^n}\int\limits_{B(
x_0,\,\varepsilon)} {\varphi}(x)\ dm(x).$
Функции с конечным средним колебанием введены А. Игнатьевым и В.
Рязановым в работе \cite{IR}, см. также \cite[разд.~11.2]{MRSY}, и
представляют собой обобщение функций $BMO,$ ограниченного среднего
колебания по Ф. Джону - Л. Ниренбергу \cite{JN}. Напомним, что
изолированная точка $x_0$ границы $\partial D$ области $D$
называется {\it устранимой} для отображения $f,$ если существует
конечный предел $\lim\limits_{x\rightarrow x_0}\,f(x).$ Если
$f(x)\rightarrow \infty$ при $x\rightarrow x_0,$ точку $x_0$ будем
называть {\it полюсом.} Изолированная точка $x_0$ границы $\partial
D$ называется {\it существенно особой точкой} отображения
$f:D\rightarrow {\Bbb R}^n,$ если при $x\rightarrow x_0$ нет ни
конечного, ни бесконечного предела. В качестве приложений теоремы
\ref{th3.1}, приведём следующие результаты.

\begin{theorem}\label{th1} {\sl Пусть $b\in D$ и $f:D\setminus\{b\}\rightarrow {\Bbb
R}^n$ -- открытое дискретное дифференцируемое почти всюду
отображение, $f\in ACP^{-1},$ обладающее $N$ и
$N^{\,-1}$--свой\-ст\-вами Лузина. Предположим, что существует
некоторое число $\delta>0,$ такое, что при всех $x\in B(b, \delta)$
имеет место неравенство
\begin{equation}\label{eq2A}
|f(x)|\le C \left(\log\frac{1}{|x-b|}\right)^{p}\,,
\end{equation}
где $p>0$ и $C>0$ -- некоторые постоянные. Пусть, кроме того,
существует измеримая по Лебегу функция $Q:D\rightarrow[1, \infty],$
такая, что $D_f(x, b)\le Q(x)$ при почти всех $x\in D$ и $Q(x)\in
FMO(b).$ Тогда точка $b$ является для отображения $f$ либо полюсом,
либо устранимой особой точкой. Кроме того, найдётся постоянная
$\beta_n,$ зависящая только от размерности пространства $n$ и
функции $Q$ такая, что, как только вместо условия (\ref{eq2A}) имеет
место более сильное предположение:
\begin{equation}\label{eq16}
\lim\limits_{x\rightarrow
b}\frac{|f(x)|}{\left(\log\frac{1}{|x-b|}\right)^{\beta_n}}=0\,,
\end{equation}
то точка $x=b$ является устранимой для отображения $f.$ }
\end{theorem}

\begin{theorem}\label{th2} {\sl Пусть $b\in D$ и $f:D\setminus\{b\}\rightarrow {\Bbb
R}^n$ -- открытое дискретное дифференцируемое почти всюду
отображение, $f\in ACP^{\,-1},$ обладающее $N$ и
$N^{\,-1}$-свой\-ст\-вами Лузина, а $Q:D\rightarrow[0, \infty]$ --
некоторая локально интегрируемая по Лебегу функция такая, что
$D_f(x, b)\le Q(x)$ при почти всех $x\in D\setminus\{b\}.$
Предположим, что существуют некоторые числа $\delta,$ $C,$ $p>0$ и
$\varepsilon_0>0,$ $\varepsilon_0<{\rm dist}\,(b, \partial D),$
такие, что при всех $x\in B(b, \delta)\setminus \{b\}$ имеет место
неравенство
\begin{equation}\label{eq2}
|f(x)|\le C \cdot
\exp\left\{p\int\limits_{|x-b|}^{\varepsilon_0}\frac{dt}{tq_{b}^{1/(n-1)}(t)}\right\}\,.
\end{equation}
Пусть, кроме того,
$\int\limits_{0}^{\varepsilon_0}\frac{dt}{tq_{b}^{1/(n-1)}(t)}=\infty$
и
$\int\limits_{\varepsilon}^{\varepsilon_0}\frac{dt}{tq_{b}^{1/(n-1)}(t)}<\infty$
при достаточно малых $\varepsilon.$ Тогда точка $b$ является для
отображения $f$ либо полюсом, либо устранимой особой точкой. Если
вместо условия (\ref{eq2}) имеет место более сильное предположение
$$\lim\limits_{x\rightarrow b} |f(x)|\cdot
\exp\left\{-\int\limits_{|x-b|}^{\varepsilon_0}\frac{dt}{tq_{b}^{1/(n-1)}(t)}\right\}
=0\,,$$ то точка $x=b$ является для отображения $f$ устранимой
особой точкой.}
\end{theorem}

\begin{corollary}\label{cor2}
{\sl Заключения теорем \ref{th1} и \ref{th2} остаются выполненными,
если $f$ -- отображение с конечным искажением длины.}
\end{corollary}

\section{До\-ка\-за\-тель\-с\-т\-во неравенства типа Вяйсяля}

Пусть $E$ -- множество в ${\Bbb R}^n$ и $\gamma :\Delta\rightarrow
{\Bbb R}^n$ -- некоторая кривая. Обозначим
$\gamma\cap E=\gamma\left(\Delta\right)\cap E.$
Пусть кривая $\gamma$ локально спрямляема и функция длины
$l_{\gamma}(t)$ такова, как было определено в предыдущем разделе.
Полагаем
$ l\left(\gamma\cap E\right):= m_1\,(E_ {\gamma}), \quad E_ {\gamma}
= l_{\gamma}\left(\gamma ^{-1}\left(E\right)\right).$
Здесь, как обычно, $m_1\,(A)$ обозначает длину (линейную меру
Лебега) множества $A\subset {\Bbb R}.$ Заметим, что
$E_{\gamma}=\gamma_0^{\,-1}(E),$
где $\gamma _0 :\Delta _{\gamma}\rightarrow {\Bbb R}^n$ --
натуральная параметризация кривой $\gamma,$  и что
$l\left(\gamma\cap E\right) = \int\limits_{\gamma} \chi_E(x)\,|dx| =
\int\limits_{\Delta _{\gamma}} \chi _{E_\gamma }(s)\,dm_1(s),$
см. \cite[разд.~4, с.~8]{Va}. Следующее утверждение связывает
свойства функции длины локально спрямляемой кривой со свойствами
произвольного измеримого множества в ${\Bbb R}^n$ (см.
\cite[теорема~9.1]{MRSY}).

\begin{proposition}\label{pr8}{\sl\,
Пусть $E$ -- множество в области $D\subset{\Bbb R}^n,$ $n\ge 2.$
Тогда множество $E$ измеримо тогда и только тогда, когда множество
$\gamma\cap E$ измеримо для почти всех кривых $\gamma$ в $D.$ Более
того, $m(E)=0$ тогда и только тогда, когда $l(\gamma\cap E)=0$ для
почти всех кривых $\gamma$ в $D.$ }
\end{proposition}

Весьма полезным для дальнейшего исследования является следующее
замечание.

\begin{remark}\label{rem15}
Пусть $\gamma:[a, b]\rightarrow {\Bbb R}^n$ спрямляемая кривая и
велчина $S(\gamma, [a,t])$ обозначает длину кривой $\gamma|_{[a,
t]}.$ Заметим, что свойства функции $L_{\gamma, f}$ между
натуральными параметрами $l_{\gamma}(t)$ и
$l_{\widetilde{\gamma}}(t)$ (локально спрямляемых) кривых $\gamma$ и
$\widetilde{\gamma}$ таких, что $\widetilde{\gamma}=f\circ\gamma,$
не зависят от выбора $t_0\in (a, b).$ В случае замкнутой кривой
$\gamma$ мы будем считать, что $t_0=a,$ поскольку при заданном
$t_0\in (a, b)$ выполнено равенство $S(\gamma, [a, t])=S(\gamma, [a,
t_0])+l_{\gamma}(t).$ Пусть $I=[a,b].$ Для спрямляемой кривой
$\gamma:I\rightarrow {\Bbb R}^n$ определим функцию длины
$l_{\gamma}(t)$ по следующему правилу: $l_{\gamma}(t)=S\left(\gamma,
[a,t]\right).$ В дальнейшем для замкнутых кривых мы отождествляем
функции $l_{\gamma}(t)$ и $S\left(\gamma, [a,t]\right),$ если не
оговорено противное.
\end{remark}

Пусть $\alpha:[a,b]\rightarrow {\Bbb R}^n$ -- спрямляемая замкнутая
кривая в ${\Bbb R}^n,$ $n\ge 2,$ $l(\alpha)$ -- её длина. {\it
Нормальным представлением кривой $\alpha$ } называется кривая
$\alpha^0:[0, l(\alpha)]\rightarrow {\Bbb R}^n,$ такая что
$\alpha(t)=\alpha^0\left(S\left(\alpha, [a, t]\right)\right)=
\alpha^0\circ l_{\alpha}(t).$ Отметим, что такая кривая
$\alpha^{\,0}$ существует и единственна, при этом, $S\left(\alpha^0,
[0, t]\right)=t$ при $t\in [0, l(\alpha)],$ см.
\cite[теорема~2.4]{Va}.

\medskip
Далее $I$ означает открытый, замкнутый или полуоткрытый конечный
интервал числовой оси. Следующее определение может быть найдено в
\cite[п.~5, гл.~II]{Ri}.

Пусть $f:D\rightarrow {\Bbb R}^n$ -- дискретное отображение,
$\beta:I_0\rightarrow {\Bbb R}^n$ замкнутая спрямляемая кривая и
$\alpha:I\rightarrow D$ кривая такая, что $f\circ \alpha\subset
\beta.$ Если функция длины $l_{\beta}:I_0\rightarrow [0, l(\beta)]$
постоянна на некотором интервале $J\subset I,$ то $\beta$ постоянна
на $J$ и, в силу дискретности $f,$ кривая $\alpha$ также постоянна
на $J.$ Следовательно, существует единственная кривая
$\alpha^{\,*}:l_\beta(I)\rightarrow D$ такая, что
$\alpha=\alpha^{\,*}\circ (l_\beta|_I).$ Будем говорить, что
$\alpha^{\,*}$ является \index{$f$-представление кривой}{\it
$f$-пред\-став\-лением кривой $\alpha$ относительно $\beta.$ }

\medskip
Следующее утверждение содержит в себе критерий выполнения
свой\-с\-т\-ва $ACP^{\,-1}$ относительно произвольного отображения
$f$ в терминах абсолютной непрерывности соответствующих кривых.

\begin{lemma}\label{pr11}{\sl\,
Слабо нульмерное отображение $f:D\rightarrow {\Bbb R}^n$ обладает
$ACP^{\,-1}$-свой\-с\-т\-вом тогда и только тогда, когда кривая
$\gamma^{\,*}$ является спрямляемой и абсолютно непрерывной для
почти всех замкнутых кривых $\widetilde{\gamma}=f\circ\gamma.$

Тут и далее $\gamma^{\,*}$ означает $f$-пред\-став\-ление кривой
$\gamma$ по отношению к $\widetilde{\gamma}.$}
\end{lemma}

\begin{proof} {\it Необходимость.} Пусть $f$ обладает
$ACP_p^{\,-1}$-свой\-с\-т\-вом. Тогда, во-первых, $L_{\gamma,
f}^{\,-1}$ определена корректно для поч\-ти всех кривых
$\widetilde{\gamma}$ таких, что $\widetilde{\gamma}=f\circ\gamma.$
Во-вторых, кривая $\gamma^{\,*}$ является спрямляемой для $p$-почти
всех замкнутых кривых $\widetilde{\gamma}$ как только
$\widetilde{\gamma}=f\circ \gamma,$ поскольку
$(\gamma^{\,*})^{\,0}=\gamma^{\,0}$ (см. в \cite[теорема~2.6]{Va}).
Кроме того, для почти всех замкнутых кривых $\widetilde{\gamma}$ и
всех $\gamma,$ таких что $\widetilde{\gamma}=f\circ \gamma,$ мы
получаем равенство
$\gamma(t)=\gamma^{\,*}\circ
l_{\widetilde{\gamma}}(t)=\gamma^{\,0}\circ
l_{\gamma}(t)=\gamma^{\,0}\circ L_{\gamma,
f}^{\,-1}\left(l_{\widetilde{\gamma}}(t)\right).$
Полагая $l_{\widetilde{\gamma}}(t):=s,$ мы получаем
\begin{equation}\label{equa4}
\gamma^{\,*}(s)=\gamma^{\,0}\circ L_{\gamma, f}^{\,-1}(s)\,.
\end{equation}
Таким образом, кривая $\gamma^{\,*}$ абсолютно непрерывна, поскольку
$L_{\gamma, f}^{\,-1}(s)$ абсолютно непрерывна и
$|\gamma^{\,0}(t_1)-\gamma^{\,0}(t_2)|\le |t_1-t_2|$
для всех $t_1, t_2\in [0, l(\gamma)]$).

{\it Достаточность.} Согласно предположению, кривая $\gamma^{\,*}$
спрямляема для почти всех замкнутых кривых
$\widetilde{\gamma}=f\circ \gamma;$ в частности,
$\gamma^{\,*\,0}=\gamma^{\,0}.$ Более того, для таких кривых
$l_{\gamma^{\,*}}(s)=L_{\gamma, f}^{\,-1}(s)$ и функция
$L^{\,-1}_{\gamma, f}$ определена корректно. Следовательно, для
почти всех замкнутых кривых $\widetilde{\gamma}$ и всех $\gamma$
таких, что $\widetilde{\gamma}=f\circ \gamma,$ кривая $\gamma$
спрямляема и функция $L_{\gamma, f}^{\,-1}(s)$ абсолютно неперервна
(см. в \cite[теорема~1.3]{Va}). Пусть $\Gamma_1$ -- семейство всех
замкнутых кривых $\widetilde{\alpha}=f\circ\alpha$ в области $f(D)$
таких, что кривая $\alpha^{\,*}$ либо не спрямляема, либо функция
$L_{\alpha, f}^{\,-1}(s)$ не абсолютно непрерывна. Пусть $\Gamma$ --
семейство всех кривых $\widetilde{\gamma}=f\circ\gamma$ в области
$f(D),$ таких что $\gamma$ либо не локально спрямляема, либо функция
$L_{\gamma, f}^{\,-1}(s)$ не локально абсолютно непрерывна. Тогда
$\Gamma>\Gamma_1$ и, следовательно, ввиду свойства (\ref{eq32*A}),
$M(\Gamma)\le M(\Gamma_1)=0,$ что и требовалось доказать.
\end{proof}

Следующий результат доказан в монографии \cite[разд.~8, леммы~8.2 и
8.3]{MRSY} (см. также \cite[следствие~3.14 и лемма~3.20]{MRSY$_1$}).

\begin{proposition}\label{lem2.4}
{\sl\, Пусть отображение $f:D\rightarrow {\Bbb R}^n$ почти всюду
дифференцируемо и обладает $N$ и $N^{\,-1}$-свой\-ст\-вами Лузина.
Тогда найдётся не более чем счётная последовательность компактных
множеств $C_k^{\,*}\subset D,$ такая что $m(B)=0,$ где $B=D\setminus
\bigcup\limits_{k=1}^{\infty} C_k^{\,*}$ и $f|_{C_k^{\,*}}$ взаимно
однозначно и билипшицево для каждого $k=1,2,\ldots .$ Более того,
$f$ дифференцируемо при всех $x\in C_k^{\,*}$ и выполнено условие
$J(x,f)\ne 0.$}
\end{proposition}

\medskip {\it Доказательство теоремы \ref{th3.1}}. Заметим, прежде
всего, что мера Лебега множества точек ветвления отображения $f$
равна нулю, $m(B_f)=0$ (см., напр., \cite[предложение~8.4]{MRSY}).
Пусть множества $B$ и $C_{k}^{\,*}$ таковы, как указано в
предложении \ref{lem2.4}. Полагаем $B_0=B\cup B_f,$
$B_1=C_{1}^{\,*}\setminus B_f,$ $B_2=C_{2}^{\,*}\setminus
B_1\ldots\,,$
$B_k= C_{k}^{\,*}\setminus \left(\bigcup\limits_{l=1}^{k-1}B_l\cup
B_f\right).$
Таким образом, мы получим не более, чем счётное покрытие области $D$
борелевскими множествами $B_k,$ $k=0,1,\ldots,$ причём $B_l\cap
B_j=\varnothing$ при $l\ne j$ и $m(B_0)=0,$ где $B_0=D\setminus
\bigcup\limits_{k=1}^{\infty} B_k.$ Поскольку отображение $f$
обладает $N$-свойством, получаем $m(f(B_0))=0.$ По предложению
\ref{pr8} $l\left(\overline{\gamma}\cap f(B_0)\right)=0$ для п.в.
кривых $\overline{{\gamma}}$ в области $f(D).$ Следовательно,
$\widetilde{\gamma}^{\,0}(s)\not\in f(B_0)$ для почти всех замкнутых
кривых $\widetilde{\gamma}$ в области $f(D)$ и почти всех $s\in [0,
l(\widetilde{\gamma})];$ здесь $\widetilde{\gamma}^{\,0}(s)$
обозначает нормальное представление кривой $\widetilde{\gamma}(s).$
Кроме того, по лемме \ref{pr11} кривая $\gamma^{\,*},$ являющаяся
$f$-пред\-став\-лением кривой $\gamma,$ абсолютно непрерывна для
почти всех замкнутых кривых $\widetilde{\gamma}=f\circ\gamma.$ Здесь
$f$-пред\-став\-ление $\gamma^{\,*}$ кривой $\gamma$ корректно
определено для почти всех кривых $\widetilde{\gamma}=f\circ\gamma,$
поскольку по предположению $f$ -- дискретное отображение.

\medskip
Учитывая, что модуль семейства неспрямляемых кривых равен нулю (см.
\cite[следствие~6.11]{Va}), а также то, что каждая спрямляемая
кривая $\gamma:(a, b)\rightarrow {\Bbb R}^n$ ($\gamma:[a,
b)\rightarrow {\Bbb R}^n$ или $\gamma:(a, b]\rightarrow {\Bbb R}^n$)
может быть продолжена непрерывным образом до соответствующей
замкнутой кривой $\gamma:[a, b]\rightarrow {\Bbb R}^n$ (см.
\cite[теорема~3.2]{Va}), мы можем считать, что $\Gamma^{\,\prime}$
состоит только из замкнутых спрямляемых кривых.

\medskip
Пусть $\rho:[r_1, r_2]\rightarrow {\Bbb R}$ -- неотрицательная
измеримая по Лебегу функция такая, что
$\int\limits_{r_1}^{r_2}\rho(r)dr\ge 1.$ Полагаем
$$\rho^*(x)\,=\,\left\{\begin{array}{rr}
\rho(|x-x_0|)\left(\frac{D_f(x, x_0)}{J(x, f)}\right)^{1/n}, &   x\in A(r_1, r_2, x_0)\setminus B_0,\\
0,  &  x\in B_0\,.
\end{array}
\right.$$
Рассмотрим следующую функцию:
\begin{equation}\label{equa9}
\widetilde{\rho}(y)\quad=\quad\frac{1}{m}\cdot
\chi_{f\left(D\setminus B_0
\right)}(y)\sup\limits_{C}\,\sum\limits_{x\,\in\,C}\rho^{\,*}(x)\,,
\end{equation}
где $C$ пробегает все подмножества $f^{-1}(y)$ в $D\setminus B_0,$
количество элементов которых не больше $m.$ Заметим, что
\begin{equation}\label{equ5}
\widetilde{\rho}(y)\quad=\quad\frac{1}{m}\cdot
\sup\sum\limits_{i=1}^s \rho_{k_i}(y)\,,
\end{equation}
где $\sup$ в (\ref{equ5}) берётся по всем возможным наборам
$\left\{k_{i_1},\ldots,k_{i_s}\right\}$ таким, что $i\in {\Bbb N},$
$k_i\in\,{\Bbb N},$ $k_i\ne k_j$ при $i\ne j,$ всех $s\le m$ и
$$
\rho_k(y)\,=\left\{\begin{array}{rr}
\,\rho^{*}\left(f_k^{-1}(y)\right), &   y\in f(B_k),\\
0,  &  y\notin f(B_k)\,,
\end{array}
\right.$$
а каждое из отображений $f_k=f|_{B_k},$ $k=1,2,\ldots,$ инъективно.
Из (\ref{equ5}) следует, что функция $\widetilde{\rho}(y)$ является
борелевской, поскольку множества $f(B_k)$ борелевские, см.
\cite[разд.~2.3.2]{Fe}. Пусть $\beta$ -- произвольная кривая
семейства $\Gamma^{\,\prime}.$

По условию найдутся кривые $\alpha_1,\ldots,\alpha_m$ семейства
$\Gamma,$ такие, что $f\circ \alpha_j\subset \beta$ для всех
$j=1,2,\ldots, m,$ и при каждом фиксированном $x\in D$ и $t\in I$
равенство $\alpha_j(t)=x$ справедливо не более, чем при $i(x,f)$
индексах $j.$

Покажем, что  функция $\widetilde{\rho}\,\in\,{\rm }\,\,{\rm
adm}\,\Gamma^{\,\prime}\setminus \Gamma_0,$ где $M(\Gamma_0)=0.$
Пусть $\beta$ -- кривая семейства $\Gamma^{\,\prime}$ и $\beta^0:[0,
l(\beta)]\rightarrow {\Bbb R}^n$ -- нормальное представление кривой
$\beta,$ $\beta(t)=\beta^0\circ l_{\beta}(t).$ Обозначим символами
$\alpha_j^{\,*}(s):I_j\rightarrow D$ соответствующие
$f$-пред\-став\-ле\-ния кривых $\alpha_j$ относительно кривой
$\beta,$ т.е. $\alpha_j(t)=\alpha^*_j\circ l_{\beta}(t),$ $t\in
I_j,$ $f\circ \alpha^*_j\subset \beta^0,$ $j=1,2,\ldots, m.$
Обозначим
$$h_j(s)=\rho^{\,*}\left(\alpha^*_j(s)\right)\chi_{I_j}(s)\,,\quad s\in [0, l(\beta)]\,,\quad
J_s:=\{j:s\in I_j\}\,.$$
Поскольку по предположению $\beta^0(s)\not\in f(B_0)$ при почти всех
$s\in [0, l(\beta)],$ при этих же $s$ точки $\alpha^{\,*}_j(s)\in
f^{\,-1}(\beta^0(s)),$ $j\in J_s,$ являются различными ввиду условия
теоремы, что равенство $\alpha_j(t)=x$ возможно не более, чем при
$i(x, f)$ индексах $j,$ а $i(x, f)=1$ на каждом $B_k$ по построению.
Тогда, по определению функции $\widetilde{\rho}$ в (\ref{equa9}),
при почти всех $s\in [0, l(\beta)]$
\begin{equation}\label{eq6.1}
\widetilde{\rho}(\beta^0(s))\ge
\frac{1}{m}\cdot\sum\limits_{j=1}^{m} h_j(s)\,.
\end{equation}
Из соотношения (\ref{eq6.1}) получаем
$$\int\limits_{\beta}\widetilde{\rho}(y)\,|dy|=\int\limits_{0}^{l(\beta)}
\widetilde{\rho}(\beta^0(s))\,ds\ge
$$
\begin{equation}\label{eq6.2}
\ge \frac{1}{m}\cdot\sum\limits_{j=1}^{m}
\int\limits_{0}^{l(\beta)}h_j(s) ds=
\frac{1}{m}\cdot\sum\limits_{j=1}^{m}
\int\limits_{I_j}\rho^{\,*}\left(\alpha^*_j(s)\right)dm_1(s)\,.
\end{equation}
Осталось показать, что
\begin{equation}\label{eq6.4}
\int\limits_{I_j}\rho^{\,*}\left(\alpha^*_j(s)\right)dm_1(s)\ge 1
\end{equation}
для почти всех кривых $\beta\in \Gamma^{\,\prime}.$ Если
$\int\limits_{I_j}\rho^{\,*}\left(\alpha^*_j(s)\right)dm_1(s)=\infty,$
доказывать нечего. Пусть
$\int\limits_{I_j}\rho^{\,*}\left(\alpha^*_j(s)\right)dm_1(s)<\infty.$
Заметим, что
$$\int\limits_{I_j}\rho^{\,*}\left(\alpha^*_j(s)\right)dm_1(s)=$$
\begin{equation}\label{equa3}
=\int\limits_{I_j}\frac{\rho(|\alpha^*_j(s)-x_0|)}{\left|\frac{dr_j}{ds}(s_*)\right|}\left(\frac{D_f(\alpha^*_j(s),
x_0)}{J(\alpha^*_j(s), f)}\right)^{1/n}\cdot
\left|\frac{dr_j}{ds}(s_*)\right|dm_1(s)\,,
\end{equation}
где $r_j(s_*):=|\alpha^0_j(s_*)-x_0|.$ Используя равенство вида
(\ref{equa4}), мы получим, что $\alpha^*_j(s)=\alpha^0_j\circ
L_{\alpha_j, f}^{\,-1}(s)=\alpha^0_j(s_*),$ где $s_*=L_{\alpha_j,
f}^{\,-1}(s),$ откуда также $s=L_{\alpha_j, f}(s_*).$ Заметим, что
для почти всех кривых $\beta\in \Gamma^{\,\prime}$ функция
$s=L_{\alpha_j, f}(s_*)$ обладает $N^{\,-1}$-свойством ввиду того,
что $f\in ACP^{\,-1}$ (см., напр., \cite[теорема~2.10.13]{Fe});
значит, ввиду известной теоремы Пономарёва $\frac{ds}{ds_*}(s_*)\ne
0$ для почти всех $s_*\in [0, l(\alpha_j)]$ (см.
\cite[теорема~1]{Pon}). Тогда по теореме о дифференцируемости
сложной функции при почти всех $s_*\in [0, l(\alpha_j)]$
\begin{equation}\label{equa5}
\left|\frac{dr_j}{ds}(s_*)\right|=\frac{\left|\frac{dr_j}{ds_*}\right|}{\left|\frac{ds}{ds_*}\right|}\,.
\end{equation}
Путём прямых вычислений, нетрудно показать, что
%
$$\left|\frac{dr_j}{ds_*}\right|=\left|\left(\alpha^{0\,\prime}_j(s_*),
\frac{\alpha^0_j(s_*)-x_0}{|\alpha^0_j(s_*)-x_0|}\right)\right|$$
%
(отметим, что кривая $\alpha^0_j(s_*)$ дифференцируема при почти
всех $s_*\in (0, l(\alpha_j))$). С другой стороны, имеем для почти
всех $s_*,$ что
$$\beta^{\,\prime}(s_*)=f^{\,\prime}(\alpha_j^0(s_*))\alpha_j^{0\,\prime}(s_*)=
\partial_{\alpha_j^{0\,\prime}(s_*)}f(\alpha_j^0(s_*))$$ и, в то же время, $$\frac{ds}{ds_*}=L^{\,\prime}_{\alpha_j,
f}(s_*)=|\beta^{\,\prime}(s_*)|$$ (по поводу последнего равенства
см., напр., \cite[пункт~(5), теорема~1.3]{Va}). Таким образом, из
(\ref{equa5}) вытекает, что
\begin{equation}\label{equa6}
\left|\frac{dr_j}{ds}(s_*)\right|=\frac{\left|\left(\alpha^{0\,\prime}_j(s_*),
\frac{\alpha^0_j(s_*)-x_0}{|\alpha^0_j(s_*)-x_0|}\right)\right|}
{|\partial_{\alpha_j^{0\,\prime}(s_*)}f(\alpha_j^0(s_*))|}\le
\frac{1}{l_f(\alpha_j^{0}(s_*), x_0)}\,.
\end{equation}
Таким образом, из (\ref{equa3}) и (\ref{equa6}) вытекает, что
$$\int\limits_{I_j}\rho^{\,*}\left(\alpha^*_j(s)\right)dm_1(s)\ge$$
\begin{equation}\label{equa10}
\ge \int\limits_{I_j}
\frac{\rho(|\alpha^*_j(s)-x_0|)}{l_f(\alpha^*_j(s), x_0)}\cdot
l_f(\alpha^*_j(s), x_0)\cdot
\left|\frac{dr_j}{ds}(s_*)\right|dm_1(s)=
\end{equation}
$$=\int\limits_{I_j}
\rho(r_j(s_*(s)))\cdot
\left|\frac{dr_j}{ds}(s_*(s))\right|dm_1(s)\,.$$
Заметим, что функция $r_j$ абсолютно непрерывна относительно
параметра $s_*$ для почти всех кривых $\beta\in \Gamma^{\,\prime}.$
Поскольку по предположению
$\int\limits_{I_j}\rho^{\,*}\left(\alpha^*_j(s)\right)dm_1(s)<\infty$
ввиду замены переменных относительно линейной меры в интеграле
Лебега (см. \cite[теорема~3.2.6]{Fe}), из (\ref{equa10}) вытекает,
что
$$\int\limits_{I_j}\rho^{\,*}\left(\alpha^*_j(s)\right)dm_1(s)\ge$$
$$
\ge \int\limits_{I_j} \rho(r_j(s_*(s)))\cdot
\frac{dr_j}{ds}(s_*(s))dm_1(s)=\int\limits_{r_1}^{r_2}\rho(r)dr\ge
1\,,
$$
что и требовалось установить. Следовательно, из (\ref{eq6.2}) и
(\ref{eq6.4}) вытекает, что функция $\widetilde{\rho}\,\in\,{\rm
}\,\,{\rm adm}\,\Gamma^{\,\prime}\setminus \Gamma_0,$ где
$M(\Gamma_0)=0$ и, значит,
\begin{equation}\label{equ6*}
M\left(\Gamma^{\,\prime}\right)\quad\le\quad\int\limits_{f(A)}\widetilde{\rho}\,^n(y)\,\,dm(y)\,.
\end{equation}
Согласно \cite[теорема~3.2.5]{Fe} для $m=n,$ получаем, что
\begin{equation}\label{equ7*}
\int\limits_{B_k}D_f(x,
x_0)\cdot\rho^n(|x-x_0|)\,\,dm(x)\quad=\quad\int\limits_{f(A)}
\rho^n_k(y)\,dm(y)\,.
\end{equation}
Заметим также, что по неравенству Гёльдера для сумм
\begin{equation}\label{equ8*}
\left(\frac{1}{m}\cdot\sum\limits_{i=1}^{s}\rho_{k_i}(y)\right)^n\quad\le\quad
\frac{1}{m}\cdot \sum\limits_{i=1}^{s}\,\rho^n_{k_i}(y)
\end{equation}
для произвольного $1 \le s \le m$ и любого набора
$\left\{k_1,\ldots,k_s\right\}$ длины $s,$ $1\le i\le s,$ $k_i\in
{\Bbb N},$ $k_i\ne k_j,$ если $i\ne j.$ Тогда по теореме об
аддитивности интеграла Лебега, см., напр., \cite[теорема~12.3,
$\S\,12,$ разд.~I]{Sa}, из (\ref{equ6*}), (\ref{equ7*}) и
(\ref{equ8*}) получаем, что
%
$$\frac{1}{m}\cdot\int\limits_{D}D_f(x,\,x_0)\cdot\rho^n(|x-x_0|)\,\,dm(x)\quad
=\quad\frac{1}{m}\cdot\int\limits_{f\left(A\right)}\,\sum\limits_{k=1}^{\infty}
\rho_k^n(y)\,dm(y)\quad\ge$$
%
%
$$\ge\quad\frac{1}{m}\cdot\int\limits_{f\left(A\right)}
\sup\limits_{\left\{k_1,\ldots,k_s\right\},\, k_i\in {\Bbb N},\atop
k_i\ne k_j, \,\, i\ne j}\sum\limits_{i=1}^s
\rho^n_{k_i}(y)\,dm(y)\quad\ge\quad
\int\limits_{f\left(A\right)}\,\widetilde{\rho}^{\,n}(y)\,dm(y)\quad\ge$$
$$\ge M(\Gamma^{\,\prime})\,.$$
Теорема доказана. $\Box$

\medskip
\begin{remark}\label{rem1}
Утверждение теоремы \ref{th3.1}, очевидно, остаётся справедливым,
если $x_0$ является изолированной точкой границы области $D.$
\end{remark}

\section{Приложения к проблеме устранения особенностей отображений}

Аналоги утверждений, приводимых в настоящем разделе, доказаны в
работе \cite{Sev$_2$} для внутренних дилатаций $K_I(x, f).$ Здесь
рассматривались дифференцируемые почти всюду отображения, обладающие
$N,$ $N^{\,-1}$ и $ACP^{\,-1}$-свойствами, для которых их внутренняя
дилатация почти всюду удовлетворяет условию $K_I(x, f)\le Q(x).$ При
определённых условиях на функцию $Q,$ а также условиях на рост в
окрестности изолированной особой точки $b\in {\Bbb R}^n,$ было
показано, что такие отображения продолжаются в точку $b$ по
непрерывности. Как оказалось, условие $K_I(x, f)\le Q(x)$ вполне
можно ослабить до условия $D_f(x, b)\le Q(x)$ ввиду теоремы
\ref{th3.1}. С другой стороны, отметим, что (независимо от условий
на дилатации, вплоть даже до их ограниченности либо равенства
единице) даже аналитические функции на плоскости не продолжаются в
изолированную точку границы области по непрерывности без некоторого
дополнительного условия выпускания этими отображениями некоторого
множества положительной ёмкости ($\varphi(z)=\exp\{1/z\},\quad
b=0$). Однако, устранение изолированной особенности имеет место,
если вместо упомянутого ёмкостного условия потребовать, чтобы в
окрестности точки $b$ соответствующее отображение имело "достаточно
слабый"\,\, порядок роста. Рассмотрим следующие определения.

Для отображения $f:D\rightarrow{\Bbb R}^n,$ множества $E\subset D$ и
$y\in{\Bbb R}^n,$  определим {\it функцию кратности $N(y,f,E)$} как
число прообразов точки $y$ во множестве $E,$ т.е.
$N(y,f,E)={\rm card}\,\left\{x\in E: f(x)=y\right\}.$ Множество
$G\subset {\Bbb R}^n$ условимся называть {\it множеством нулевой
ёмкости,} пишем ${\rm cap\,}G =0,$ если существует континуум
$T\subset {\Bbb R}^n,$ такой что $M(\Gamma(T, G, {\Bbb R}^n))=0,$
см., напр., \cite[разд.~2 гл.~III и предложение 10.2 гл.~II]{Ri}. В
противном случае мы говорим, что $G$ имеет положительную ёмкость,
что записываем как ${\rm cap\,}G >0.$ Множества ёмкости нуль, как
известно, всюду разрывны (любая компонента их связности вырождается
в точку), т.е., условие ${\rm cap\,}G =0$ влечёт, что ${\rm
Int\,}G=\varnothing,$ см., напр., \cite[следствие~2.5, гл.~III]{Ri}.
Открытое множество $U\subset D,$ $\overline{U}\subset D,$ называется
{\it нормальной окрестностью} точки $x\in D$ при отображении
$f:D\rightarrow {\Bbb R}^n,$ если $U\cap
f^{\,-1}\left(f(x)\right)=\left\{x\right\}$ и $\partial
f(U)=f(\partial U),$ см., напр., \cite[разд.~4, гл.~I]{Ri}.

\begin{proposition}\label{pr2}{\sl\, Пусть $f:D\rightarrow {\Bbb R}^n$ открытое дискретное отображение.
Тогда для каждого $x\in D$ существует $s_x,$ такое, что при всех
$s\in (0, s_x)$ компонента связности множества $f^{-1}\left(B(f(x),
s)\right),$ содержащая точку $x,$ и обозначаемая символом
$U(x,f,s),$ является нормальной окрестностью точки $x$ при
отображении $f,$ при этом $f\left(U(x,f,s)\right)=B(f(x), s)$ и
$d(U(x,f,s))\rightarrow 0$ при $s\rightarrow 0.$ (Здесь, как и
прежде, $d(A)$ обозначает евклидов диаметр множества $A\subset {\Bbb
R}^n$).}
\end{proposition}

Важную роль при доказательстве основных результатов работы играют
следующее утверждение, см. \cite[лемма~5.1]{SS}.

\medskip
\begin{proposition}\label{pr4}{\sl\,Пусть $Q:{\Bbb B}^n\setminus
\left\{0\right\}\rightarrow [0, \infty]$ -- измеримая по Лебегу
функция, $f:{\Bbb B}^n\setminus \left\{0\right\} \rightarrow
\overline{{\Bbb R}^n},$ $n \ge 2\,,$ -- открытое дискретное
отоб\-ра\-же\-ние, удовлетворяющее неравенству
\begin{equation}\label{eq2AA}
M(f(\Gamma))\le \int\limits_{D} Q(x)\cdot \eta^n (|x-x_0|)\,dm(x)\,,
\end{equation}
где $\Gamma:=\Gamma(S(0, r_1), S(0, r_2), A(r_1, r_2, 0)),$
$0<r_1<r_2<1$ и $\eta:[r_1, r_2]\rightarrow {\Bbb R}$ --
произвольная неотрицательная измеримая по Лебегу функция,
удовлетворяющая условию $\int\limits_{r_1}^{r_2}\eta(r)dr\ge 1.$
Пусть, кроме того, ${\rm cap}\,\left(\overline{{\Bbb
R}^n}\setminus\,f\left({\Bbb
B}^n\setminus\left\{0\right\}\right)\right)>0.$ Предположим, что
существует $\varepsilon_0\in(0,1)$ такое, что при
$\varepsilon\rightarrow 0$
$\int\limits_{\varepsilon<|x|<\varepsilon_0}Q(x)\cdot\psi^n(|x|) \
dm(x)\,=\,o\left(I^n(\varepsilon, \varepsilon_0)\right),$
где $\psi(t)$ -- неотрицательная на $(0,\infty)$ функция, такая что
$\psi(t)>0$ п.в. и $
0<I(\varepsilon,
\varepsilon_0)=\int\limits_{\varepsilon}^{\varepsilon_0}\psi(t)dt <
\infty
$
для всех $\varepsilon\in(0, \varepsilon_0).$ Тогда $f$ имеет
непрерывное продолжение $f:{\Bbb B}^n\rightarrow\overline{{\Bbb
R}^n}$ в ${\Bbb B}^n.$ }
\end{proposition}

Следующее определение может быть найдено в \cite[гл.~II, п.~3]{Ri}.
Пусть $f:D \rightarrow {\Bbb R}^n$, $n\ge 2,$ -- отображение,
$\beta: [a,\,b)\rightarrow {\Bbb R}^n$ -- некоторая кривая и
$x\in\,f^{\,-1}\left(\beta(a)\right).$ Кривая $\alpha:
[a,\,c)\rightarrow D$ называется {\it максимальным поднятием} кривой
$\beta$ при отображении $f$ с началом в точке $x,$ если $(1)\quad
\alpha(a)=x;$ $(2)\quad f\circ\alpha=\beta|_{[a,\,c)};$ $(3)$\quad
если $c<c^{\prime}\le b,$ то не существует кривой $\alpha^{\prime}:
[a,\,c^{\prime})\rightarrow D,$ такой что
$\alpha=\alpha^{\prime}|_{[a,\,c)}$ и $f\circ
\alpha=\beta|_{[a,\,c^{\prime})}.$ Следующая конструкция является
обобщением приведённого выше определения. Пусть $x_1,\ldots,x_k$ --
$k$ различных точек множества $f^{-1}\left(\beta(a)\right)$ и
\begin{equation}\label{eq5.8A}
\widetilde{m} = \sum\limits_{i=1}^k i(x_i,\,f)\,.
\end{equation}
Последовательность кривых $\alpha_1,\dots,\alpha_{\widetilde{m}}$
является {\it максимальной последовательностью поднятий кривой
$\beta$ при отображении $f$ с началом в точках $x_1,\ldots,x_k,$}
если

$(a)$\quad каждая кривая $\alpha_j$ является максимальным поднятием
кривой $\beta$ при отображении $f,$

$(b)\quad {\rm card}\,\left\{j:a_j(a)=x_i\right\}= i(x_i,\,f),\quad
1\le i\le k\,,$

$(c)\quad {\rm card}\,\left\{j:a_j(t)=x\right\}\le i(x,\,f)$ при
всех $x\in D$ и всех $t\in I_j,$ где $I_j$ -- область определения
кривой $\alpha_j.$ Отметим, что количество кривых $\widetilde{m}$
может быть больше, чем количество соответствующих точек $k,$ см.
соотношение (\ref{eq5.8A}). Следующее утверждение см., напр., в
\cite[теорема~3.2 гл.~II]{Ri}.

\medskip
\begin{proposition}\label{pr4.2.1}
{\sl\, Пусть $f$ -- открытое дискретное отображение и точки
$x_1,\ldots,x_k\in\,f^{-1}\left(\beta(a)\right).$ Тогда кривая
$\beta$ имеет максимальную последовательность поднятий при
отображении $f$ с началом в точках $x_1,\ldots,x_k.$}
\end{proposition}

\noindent Следующая лемма является фундаментальным утверждением
на\-с\-то\-я\-ще\-го раздела.

\begin{lemma}\label{lem4.4.1}{\sl\, Пусть $b\in D$ и
$f:D\setminus\{b\}\rightarrow {\Bbb R}^n$ -- открытое дискретное
дифференцируемое почти всюду отображение, $f\in ACP^{-1},$
обладающее $N$ и $N^{\,-1}$-свой\-ст\-вами Лузина. Предположим, что
существует некоторое число $\delta>0,$ такое, что при всех $x\in
B(b, \delta)\setminus \{b\}$ и некоторой строго убывающей функции
$\varphi:(0, \infty)\rightarrow (0, \infty),$ для которой
$\varphi(t)\rightarrow \infty$ при $t\rightarrow 0,$ имеет место
неравенство
\begin{equation}\label{eq4.4.1}
|f(x)|\le C \cdot \varphi^p (|x-b|)\,,
\end{equation}
где $p>0$ и $C>0$ -- некоторые постоянные. Пусть, кроме того,
существуют измеримая по Лебегу функция $Q:D\rightarrow [0, \infty],$
числа $\varepsilon_0>0,$ $\varepsilon_0<{\rm dist\,}\left(b,
\partial D\right),$ $\varepsilon_0^{\,\prime}\in (0, \varepsilon_0),$ $A>0$ и
измеримая по Лебегу функция $\psi(t):(0, \varepsilon_0)\rightarrow
[0, \infty],$ $\psi(t)>0$ п.в., такие, что $D_f(x, x_0)\le Q(x)$ при
почти всех $x\in B(0, \delta)\setminus\{b\}$ и
$\varepsilon\rightarrow 0$
\begin{equation} \label{eq4.4.2}
\int\limits_{\varepsilon<|x-b|<\varepsilon_0}Q(x)\cdot\psi^n(|x-b|)
\ dm(x)\le \frac{A\cdot I^n(\varepsilon,
\varepsilon_0)}{\left(\log\varphi(\varepsilon)\right)^{n-1}}\,,
\end{equation} где
\begin{equation} \label{eq11}
I(\varepsilon, \varepsilon_0):
=\int\limits_{\varepsilon}^{\varepsilon_0}\psi(t)dt < \infty \qquad
\forall\quad\varepsilon \in(0,
\varepsilon_0^{\,\prime})\end{equation}
и, кроме того, $I(\varepsilon, \varepsilon_0)\rightarrow\infty$ при
$\varepsilon\rightarrow 0.$ Тогда точка $b$ является для отображения
$f$ либо полюсом, либо ус\-т\-ра\-ни\-мой особой точкой.}
\end{lemma}

\begin{proof} Предположим противное, а именно, что точка $b$
является существенно особой точкой отображения $f.$ Не ограничивая
общности рассуждений, можно считать, что $b=0,$ $\delta<{\rm
dist\,}(0, \partial D)$ и $C=1.$ В таком случае, сфера $S(0,
\delta)$ является компактным множеством в $D\setminus\{0\},$ поэтому
найдётся $R>0,$ такое что
\begin{equation}\label{eq4.4.3}
f\left(S(0, \delta)\right)\subset B(0, R)\,.
\end{equation}
В силу теоремы \ref{th3.1} отображение $f$ удовлетворяет оценке
(\ref{eq2AA}) в точке $b=0.$ Поскольку $b=0$ является существенно
особой точкой отображения $f,$ в виду условий (\ref{eq4.4.2}) и
(\ref{eq11}), по предложению \ref{pr4} отображение $f$ в $B(0,
\delta)\setminus \{0\}$ принимает все значения в ${\Bbb R}^n,$ за
исключением, может быть, некоторого множества ёмкости нуль, т.е.,
$N\left(y, f, {\Bbb B}^n\setminus \{0\}\right)=\infty$ при всех
$y\in {\Bbb R}^n\setminus E,$ где ${\rm cap\,}E=0.$ Так как $E$
имеет ёмкость нуль, множество ${\Bbb R}^n\setminus E$ не может быть
ограниченным. В таком случае, найдётся $y_0\in {\Bbb R}^n\setminus
\left(E\cup B(0, R)\right).$

Пусть $k_0>\frac{4Ap^{n-1}}{\omega_{n-1}},$ $k_0\in {\Bbb N}.$
Поскольку $N\left(y_0, f, {\Bbb B}^n\setminus \{0\}\right)=\infty,$
найдутся точки $x_1,\ldots,x_{k_0}\in f^{-1}(y_0)\cap \left(B(0,
\delta)\setminus\{0\}\right).$ По предложению \ref{pr2} при
некотором фиксированном $r>0$ каждая точка $x_j,$ $j=1,\ldots,k_0,$
имеет нормальную окрестность $U_j:=U(x_j, f , r),$ такую, что
$\overline{U_l}\cap\overline{U_{\widetilde{m}}}=\varnothing$ при
всех $l\ne \widetilde{m}, $ $l, \widetilde{m}\in {\Bbb N},$ $1\le
l\le k_0$ и $1\le \widetilde{m}\le k_0.$

Полагаем $d:=\min\left\{\varepsilon_0, {\rm dist\,}\left(0,
\overline{U_1}\cup\ldots\cup \overline{U_{k_0}}\right)\right\}.$
Пусть $a\in (0, d)$ и $V:=B(0, \delta)\setminus\overline{B(0, a)}.$
В силу неравенства (\ref{eq4.4.1}), строгого убывания функции
$\varphi,$ а также предположения о том, что $C=1,$ имеем
\begin{equation}\label{eq4.4.4}
f(V)\subset B\left(0, \varphi^p(a)\right)\,.
\end{equation}
Поскольку $z_0:=y_0+re\in \overline{B(y_0,
r)}=f\left(\overline{U(x_j, f, r)}\right),$ $j=1,\ldots, k_0,$ где
$e$ -- единичный вектор, найдётся конечная последовательность точек
$\widetilde{x_1},\ldots,\widetilde{x_{k_0}},$ $\widetilde{x_j}\in
\overline{U_j},$ $1\le j\le k_0,$ такая что
$f(\widetilde{x_j})=z_0.$ Заметим, что
$k_0\le\sum\limits_{j=1}^{k_0} i(\widetilde{x_j},\,f)=m^{\,\prime}.$
Заметим, что $z_0\in f(V).$ Обозначим через $H$ полусферу
$H=\left\{e\in {\Bbb S}^{n-1}: (e, y_0)>0\right\},$
через $\Gamma^{\,\prime}$ -- семейство всех кривых $\beta:\left[r,
\varphi^p(a)\right)\rightarrow {\Bbb R}^n$ вида $\beta(t)=y_0+te,$
$e\in H,$ $t\in [r, \varphi^p(a)),$ а через $\Gamma$ максимальную
последовательность под\-ня\-тий кривой $\beta$ при отображении $f$
относительно области $V$ с началом в точках
$\widetilde{x_1},\ldots,\widetilde{x_{k_0}},$ $\widetilde{x_j}\in
\overline{U_j},$ $1\le j\le k_0,$ состоящую из $m^{\,\prime}$
кривых, $m^{\,\prime} =\sum\limits_{j=1}^{k_0}
i(\widetilde{x_j},\,f),$ которая существует в силу предложения
\ref{pr4.2.1}. Заметим, что ввиду (\ref{eq4.4.4}) при любом
фиксированном $e\in H,$ каждой кривой $\beta=y_0+te$ и каждого
максимального её поднятия $\alpha(t):[r, c)\rightarrow V$ с началом
в точке $\widetilde{x_{j_0}},$ $\alpha\in \Gamma,$  $1\le j_0\le
k_0,$ существует последовательность $r_k\in [r, c),$ такая что
$r_k\rightarrow c-0$ при $k\rightarrow \infty$ и ${\rm
dist\,}(\alpha(r_k),
\partial V)\rightarrow 0$ при $k\rightarrow \infty.$
Кроме того, заметим, что ситуация, когда ${\rm dist\,}(\alpha(r_k),
S(0, \delta))\rightarrow 0$ при $k\rightarrow \infty,$ исключена.
Действительно, пусть эта ситуация имеет место. Тогда найдутся
$p_2\in S(0, \delta)$ и подпоследовательность номеров $k_l,$ $l\in
{\Bbb N},$ такие, что $\alpha(r_{k_l})\rightarrow p_2$ при
$l\rightarrow \infty.$ Отсюда, по непрерывности $f,$ получаем, что
$\beta(r_{k_l})\rightarrow f(p_2)$ при $l\rightarrow \infty,$ что
невозможно ввиду соотношения (\ref{eq4.4.3}), поскольку, при каждом
фиксированном $e\in H$ и $t\in \left[r, \varphi^p(a)\right),$ имеем
$|\beta(t)|=|y_0+te|=\sqrt{|y_0|^2 + 2t(y_0, e)+t^2}\ge |y_0|> R$ по
выбору $y_0.$

Из сказанного выше следует, что найдётся последовательность $r_k\in
[r, c),$ такая что $r_k\rightarrow c-0$ при $k\rightarrow \infty,$ и
$\alpha(r_k)\rightarrow p_3\in S(0, a).$ Кроме того, каждая такая
кривая $\alpha\in \Gamma$ пересекает сферу $S(0, d),$ поскольку,
согласно построению, $\alpha$ имеет начало вне шара $B(0, d).$

Из сказанного следует, что при всех достаточно малых $\varepsilon$
кривая $\alpha$ содержит замкнутую подкривую $\alpha^{\,\prime},$
пересекающую сферы $S(0, d)$ и $S(0, a+\varepsilon).$ Тогда по
теореме \ref{th3.1} и с учётом замечания \ref{rem1}
$$M(\Gamma^{\,\prime} )\le
\frac{1}{m^{\,\prime}}\quad\int\limits_D D_f\left(x,\,0\right)\cdot
\rho^n (|x|) dm(x)\le$$
\begin{equation}\label{eq4.4.5} \le \frac{1}{k_0}\int\limits_D
D_f(x, 0)\cdot \rho^n (|x|) dm(x)
\end{equation}
для каждой измеримой по Лебегу функции $\rho$ такой, что
$\int\limits_{a+\varepsilon}^{d}\rho(t)dt\ge 1.$
Из условия $I(a, d)\rightarrow \infty$ при $a\rightarrow 0$
вытекает, что $I(a, d)>0$ при малых $a.$ Рассмотрим функцию
$$\rho_{a, \varepsilon}(t)= \left\{
\begin{array}{rr}
\psi(t)/I(a+\varepsilon, d), &   t\in (a+\varepsilon, d),\\
0,  &  t\in {\Bbb R}^n \setminus (a+\varepsilon, d)
\end{array}
\right.\,, $$
где величина $I(a+\varepsilon, d)$ определена также, как в
(\ref{eq11}), а $\psi$ -- функция из условия леммы. Заметим, что
$$\int\limits_{a+\varepsilon}^d\rho_{a, \varepsilon}(t)dt=\frac{1}{I(a+\varepsilon, d)}
\int\limits_{a+\varepsilon}^d \psi(t)dt=1\,,$$ в таком случае, ввиду
условий (\ref{eq4.4.2}) и (\ref{eq4.4.5}) получаем, что
$$M(\Gamma^{\,\prime} )\quad\le\quad \frac{1}{k_0\cdot I^n(a+\varepsilon,
d)}\quad\int\limits_{a+\varepsilon<|x|<d} D_f(x, 0)\cdot \psi^n
(|x|)dm(x)\le$$
\begin{equation}\label{eq4.4.8}\le \frac{2}{k_0\cdot I^n(a+\varepsilon,
\varepsilon_0)}\quad\int\limits_{a+\varepsilon<|x|<\varepsilon_0}
Q(x)\cdot \psi^n (|x|)dm(x)\end{equation} при всех $a+\varepsilon\in
(0, d_1)$ и некотором $d_1,$ $d_1\le d,$ поскольку
$I^n(a+\varepsilon,d)\rightarrow\infty$ при
$a+\varepsilon\rightarrow 0.$ Снова, из (\ref{eq4.4.2}) и
(\ref{eq4.4.8}) получаем, что при $a+\varepsilon\in (0, d_1)$
\begin{equation}\label{eq12}
M(\Gamma^{\,\prime} )\le
\frac{2A}{k_0\left(\log\varphi(a+\varepsilon)\right)^{n-1}}\,.
\end{equation}
С другой стороны, в силу \cite[замечание~7.7]{Va},
\begin{equation}\label{eq13}
M(\Gamma^{\prime})=\frac{1}{2}\frac{\omega_{n-1}}
{\left(\log\frac{\varphi^p(a+\varepsilon)}{r}\right)^{n-1}}\,.
\end{equation}
Тогда из неравенства (\ref{eq12}) и равенства (\ref{eq13}) получаем,
что
$$\frac{1}{2}\frac{\omega_{n-1}}
{\left(\log\frac{\varphi^p(a+\varepsilon)}{r}\right)^{n-1}}\le
\frac{2A}{k_0\left(\log\varphi(a+\varepsilon)\right)^{n-1}}\,,$$
откуда
$$\frac{1}{r^{{
\left(\frac{2}{\omega_{n-1}}\right)^{\frac{1}{n-1}}}}}\ge
\left(\varphi(a+\varepsilon)\right)^{{\left(\frac{k_0}{2A}\right)^{\frac{1}{n-1}}}-p{
\left(\frac{2}{\omega_{n-1}}\right)^{\frac{1}{n-1}}}}\,.$$ Поскольку
по выбору $k_0>\frac{4Ap^{n-1}}{\omega_{n-1}},$ в правой части
последнего соотношения величина $\varphi(a+\varepsilon)$ берётся в
некоторой положительной степени. Переходя здесь к пределу при
$a+\varepsilon\rightarrow 0$ и учитывая, что по условию леммы
$\varphi(a+\varepsilon)\rightarrow\infty$ при
$a+\varepsilon\rightarrow 0,$ получаем, что
$$\frac{1}{r^{{
\left(\frac{2}{\omega_{n-1}}\right)^{\frac{1}{n-1}}}}}\ge
\infty\,,$$ что невозможно. Полученное противоречие означает, что
точка $b=0$ не может быть существенно особой для отображения $f.$
\end{proof}

\medskip
\medskip\, Отдельный случай леммы \ref{lem4.4.1} представляет собой
ситуация, когда $I(\varepsilon, \varepsilon_0)$ $\le M\cdot
\log\varphi(\varepsilon)$ при некоторой постоянной $M>0$ и
$\varepsilon\rightarrow 0.$ Покажем, что в этом случае при указанных
в формулировке леммы \ref{lem4.4.1} отображениях, предполагающихся
ограниченными, имеет место явная оценка искажения хордального
расстояния.

\medskip
Следующее утверждение может быть получено как следствие из
\cite[лемма~3.3]{Sev$_3$} и оценки (\ref{eq5}) при $m=1.$

\medskip
\begin{lemma}\label{pr4.4.1}
{\sl\, Предположим, что $b\in D,$ $f:D\rightarrow B(0, R)$ --
открытое дискретное дифференцируемое почти всюду отображение, $f\in
ACP^{-1},$ обладающее $N$ и $N^{\,-1}$-свой\-ст\-вами Лузина, при
этом, существуют измеримая по Лебегу функция $Q:D\rightarrow [1,
\infty],$ числа $\varepsilon_0>0,$ $\varepsilon_0<{\rm
dist\,}\left(b,
\partial D\right),$ и $A>0,$ такие, что $D_f(x, b)\le Q(x)$ почти всюду в $D,$ при этом,
при $\varepsilon\rightarrow 0$ имеют место соотношения
(\ref{eq4.4.2})--(\ref{eq11}). Пусть, кроме того, существует
постоянная $M>0$ и $\varepsilon_1>0,$ $\varepsilon_1\in (0,
\varepsilon_0),$ такие что при всех $\varepsilon\in (0,
\varepsilon_1)$ выполнено условие
\begin{equation} \label{eq4.4.9}
I(\varepsilon, \varepsilon_0)\le M\cdot\log\varphi(\varepsilon)\,,
\end{equation}
где $I(\varepsilon, \varepsilon_0)$ определяется соотношением
(\ref{eq11}), а $\varphi:(0, \infty)\rightarrow [0, \infty)$ --
некоторая функция. Тогда при всех $x\in B(b, \varepsilon_1)$ имеет
место оценка
\begin{equation}\label{eq15}
|f(x)-f(b)|\le \frac{\alpha_n(1+R^2)}{\delta}\exp\{-\beta_n
I\left(|x-b|, \varepsilon_0\right)\}\,,
\end{equation}
где постоянные $\alpha_n$ и
$\beta_n=\left(\frac{\omega_{n-1}}{AM^{n-1}}\right)^{1/(n-1)}$
зависят только от $n,$ а $\delta$ -- от $R.$ }
\end{lemma}

\begin{proof} В первую очередь, заметим, что $f$ удовлетворяет соотношению вида
(\ref{eq5}) при $m=1.$ Из соотношения (\ref{eq4.4.2}) с учётом
(\ref{eq4.4.9}) следует, что при $\varepsilon\in (0, \varepsilon_1)$
\begin{equation}\label{eq1AAA}
\int\limits_{\varepsilon<|x-b|<\varepsilon_0}Q(x)\cdot\psi^n(|x-b|)
\ dm(x)\le AM^{n-1}\cdot I(\varepsilon, \varepsilon_0)\,.
\end{equation}
Поскольку $|f(x)-f(b)|\le (1+R^2)\cdot h(f(x), f(b)),$ из
(\ref{eq1AAA}) и \cite[лемма~5]{Sev$_3$} вытекает соотношение
(\ref{eq15}).
\end{proof}

\medskip
\, Мы показали, что при определённых условиях изолированная
особенность отображений, более общих, чем отображения с конечным
искажением длины, является либо полюсом, либо устранимой особой
точкой. Однако, как мы увидим ниже, при ещё более сильных
ограничениях на рост отображения $f$ ситуация, когда изолированная
точка границы является полюсом, также исключена. Подобный результат
может быть получен как следствие из оценки расстояния (\ref{eq15}).
Имеет место следующее утверждение.

\medskip
\begin{corollary}\label{cor4.4.1}{\sl\, Предположим, что в условиях леммы \ref{lem4.4.1},
помимо соотношений (\ref{eq4.4.2}) и (\ref{eq11}) имеет место
условие (\ref{eq4.4.9}), а вместо условия (\ref{eq4.4.1}) имеет
место более сильное предположение:
\begin{equation}\label{eq16AA}
\lim\limits_{x\rightarrow b}|f(x)|\cdot \exp\{-\beta_n I\left(|x-b|,
\varepsilon_0\right)\}=0\,,
\end{equation}
где $\beta_n=\left(\frac{\omega_{n-1}}{AM^{n-1}}\right)^{1/(n-1)}.$
Тогда точка $x=b$ является устранимой изолированной особой точкой
отображения $f.$ }
\end{corollary}

\begin{proof}
Можно считать, что $b=0.$ По лемме \ref{lem4.4.1}, точка $b$ не
может быть существенно особой для $f.$ Предположим, что $b=0$
является для отображения $f$ полюсом. Тогда рассмотрим композицию
отображений $h=g\circ f,$ где $g(x)=\frac{x}{|x|^2}$ -- инверсия
относительно единичной сферы ${\Bbb S}^{n-1}.$ Заметим, что $h\in
ACP^{-1},$ обладает $N$ и $N^{\,-1}$-свой\-ст\-вами Лузина, при
этом, $D_f(x, 0)=D_h(x, 0)$ и $h(0)=0.$ Кроме того, в некоторой
окрестности нуля отображение $h$ (по построению) является
ограниченным. В таком случае, найдутся $\varepsilon_0>0$ и $R>0,$
такие, что $|h(x)|\le R$ при $|x|<\varepsilon_0.$ Следовательно,
возможно применение леммы \ref{pr4.4.1}. По неравенству
(\ref{eq15}), $|h(x)|=\frac{1}{|f(x)|}\le
\frac{\alpha_n(1+R^2)}{\delta}\exp\{-\beta_n I\left(|x|,
\varepsilon_0\right)\}.$ Отсюда следует, что
$$|f(x)|\cdot \exp\{-\beta_n I\left(|x|,
\varepsilon_0\right)\}\ge \frac{\delta}{\alpha_n(1+R^2)}\,.$$
Однако, последнее соотношение противоречит (\ref{eq16AA}).

Полученное противоречие доказывает, что точка $b=0$ является
устранимой для отображения $f.$
\end{proof}

Сформулируем и докажем теперь основные результаты настоящего
раздела. Имеет место следующее утверждение.

\begin{theorem}\label{cor4.4.2}{\sl\, Пусть $b\in D$ и
$f:D\setminus\{b\}\rightarrow {\Bbb R}^n$ -- открытое дискретное
дифференцируемое почти всюду отображение, $f\in ACP^{-1},$
обладающее $N$ и $N^{\,-1}$-свой\-ст\-вами Лузина, при этом $D_f(x,
b)\le Q(x)$ почти всюду и для некоторых $\varepsilon_0>0,$
$\varepsilon_0<{\rm dist\,}(b,
\partial D)$ и $\varepsilon_0^{\,\prime}\in (0, \varepsilon_0)$ имеет место условие
\begin{equation}\label{eq14}
\int\limits_{\varepsilon<|x-b|<\varepsilon_0}\frac{Q(x)}{|x-b|^n
\log^n\frac{1}{|x-b|}}\ dm(x)\le A\cdot\log{\frac{\log{\frac{1}
{\varepsilon}}}{\log{\frac{1}{\varepsilon_0}}}}\qquad\forall\,\,\varepsilon\in
(0, \varepsilon_0^{\,\prime})\,,
\end{equation}
кроме того,
\begin{equation}\label{eq16A}
\lim\limits_{x\rightarrow
b}\frac{|f(x)|}{\left(\log\frac{1}{|x-b|}\right)^{\beta_n}}=0\,,
\end{equation}
где $\beta_n=\left(\frac{\omega_{n-1}}{A}\right)^{1/(n-1)}.$ Тогда
точка $x=b$ является устранимой для отображения $f.$ }
\end{theorem}

\medskip
\begin{proof} Полагаем $\varphi(t):=\log\frac{1}{t}$ и $\psi(t):=\frac{1}{t\log\frac{1}{t}}.$
Отметим, что в этом случае выполнено соотношение (\ref{eq4.4.9}) при
$M=1$ и соотношение (\ref{eq16AA}), которое соответствует
соотношению (\ref{eq16A}) при указанном выборе функций $\varphi$ и
$\psi$ (где, как и прежде, $I(\varepsilon,
\varepsilon_0)=\int\limits_{\varepsilon}^{\varepsilon_0}\psi(t)dt$).
Кроме того, выполнены соотношения (\ref{eq4.4.2})--(\ref{eq11}).
Тогда необходимое заключение следует из следствия \ref{cor4.4.1}.
$\Box$
\end{proof}

Для дальнейшего изложения крайне важным является следующее
утверждение (см. \cite[следствие~2.3]{IR}, см. также
\cite[лемма~6.1, гл.~6]{MRSY}).

\begin{proposition}\label{pr4*!}{\sl\,
Пусть $\varphi: D\rightarrow {\Bbb R}, n\ge 2,$ -- неотрицательная
функция, имеющая конечное среднее колебание в точке $0\in D $. Тогда
\begin{equation}\label{equa17}
\int\limits_{\varepsilon<|x|< {\varepsilon_0}}\frac{\varphi (x)\,
dm(x)} {\left(|x| \log \frac{1}{|x|}\right)^n} = O \left(\log\log
\frac{1}{\varepsilon}\right)
\end{equation}
при $\varepsilon \rightarrow 0 $ и для некоторого $\varepsilon_0>0,$
$\varepsilon_0 \le {\rm dist}\,\left(0,\partial D\right).$
}
\end{proposition}

{\it Доказательство теоремы \ref{th1}.} Выберем в лемме
\ref{lem4.4.1} в качестве $\varphi(t):=\log\frac{1}{t}$ и
$\psi(t):=\frac{1}{t\log\frac{1}{t}}.$ Тогда доказательство теоремы
\ref{th1} вытекает из леммы \ref{lem4.4.1} и оценки (\ref{eq14}),
справедливой для некоторого $\varepsilon_0>0$ и произвольной функции
$Q\in FMO(b)$ (см. предложение \ref{pr4*!}), а также леммы
\ref{lem4.4.1}. Первая часть теоремы \ref{th1} доказана.

Доказательство второй части теоремы \ref{th1} немедленно следует из
теоремы \ref{cor4.4.2}, поскольку, как уже было отмечено выше,
условия вида (\ref{eq14}) выполняются для произвольных функций
класса $FMO$ в соответствующей точке. $\Box$

\medskip
{\it Доказательство теоремы \ref{th2}.} В лемме \ref{lem4.4.1}
полагаем
$$\psi(t)=1/tq_b^{1/(n-1)}(t), \quad \varphi(t)=
\exp\left\{\int\limits_{t}^{\varepsilon_0}\frac{dr}{rq_{b}^{1/(n-1)}(r)}\right\}\,.$$
Поскольку функция $Q(x)$ по условию локально интегрируема, по
теореме Фубини $q_b(r)<\infty$ при почти всех $r\in (0,
\varepsilon_0),$ откуда вытекает строгое убывание функции $\varphi$
и положительность функции $\psi.$ Кроме того, по теореме Фубини
имеем
$\int\limits_{\varepsilon<|x-b|<\varepsilon_0}Q(x)\cdot\psi^n(|x-b|)
dm(x)=\int\limits_{\varepsilon}^{\varepsilon_0}\int\limits_{S(b, r)}
Q(x)\cdot\psi^n(|x-b|)\,d{\mathcal{H}^{n-1}}\,dr =
\omega_{n-1}\cdot\int\limits_{\varepsilon}^{\varepsilon_0}r^{n-1}\psi^n(r)
q_b(r)dr = \omega_{n-1}\cdot I(\varepsilon, \varepsilon_0)=
\omega_{n-1}\cdot\log\varphi(\varepsilon).$
Отсюда, в частности, следует, что выполнено условие (\ref{eq4.4.9})
при $M=1.$ Оставшаяся часть утверждения следует из леммы
\ref{lem4.4.1} и следствия \ref{cor4.4.1}. $\Box$

\section{Некоторые примеры и замечания}
В качестве приложений полученных в работе результатов, укажем, в
частности, на один из подклассов отображений с конечным искажением
длины, для которых утверждения теорем \ref{th3.1}--\ref{th2} имеют
место (см. \cite[теорема~1]{Sev$_4$}).
\begin{theorem}\label{th3}
{\sl\,Утверждения теорем \ref{th3.1}--\ref{th2} выполняются, если
отображение $f\in W_{loc}^{1, n}(D)$ является открытым, дискретным,
мера его множества точек ветвления равна нулю и, кроме того,
внутренняя дилатация $K_I(x, f)$ отображения $f$ является локально
суммируемой в области $D.$ В частности, заключения теорем
\ref{th3.1}--\ref{th2} выполнены, если $f$ -- отображение с
ограниченным искажением.}
\end{theorem}
Заметим, что дилатация $D_f(x, x_0),$ определённая соотношением
(\ref{equa11}), заведомо не может обслуживать все семейства кривых в
области $D$ подобно соотношению (\ref{equa1}) (т.е., в (\ref{equa1})
величина $K_I(x, f)$ не может быть заменена на $D_f(x, x_0)$), так
как даже при $m=1,$ в этом случае, должно быть $D_f(x, x_0)\ge 1$
почти всюду (см. \cite[Следствие~4.1]{SevSal}). В то же время,
$D_f(x, x_0)$ может быть меньше единицы на множестве положительной
меры. По этому поводу, рассмотрим для простоты случай $n=2.$
Согласно \cite[лемма~2.10]{RSY} для дифференцируемого и
невырожденного отображения $f:D\rightarrow {\Bbb C}$ в точке $z\in
D\subset {\Bbb C}$
$$D_f(z, z_0)=\frac{\left| 1-
\frac{\overline{z-z_0}}{z-z_0}\mu(z)\right|^2}{1-|\mu(z)|^2}\,,$$
где, как обычно, $z=x + iy,$ $\overline{\partial} f=
f_{\overline{z}} = \left(f_x + if_y\right)/2$ и $\partial f = f_z =
\left(f_x - if_y\right)/2$ и $\mu(z)=\mu_f(z)=f_{\overline{z}}/f_z,$
когда $f_z \ne 0,$ и $\mu(z)=0$ в противном случае. Обозначим через
$\frak{F}_Q$ класс всех $Q$-ква\-зи\-кон\-фор\-м\-ных автоморфизмов
$f$ расширенной комплексной плоскости, нормированных условиями
$f(0)=0,$ $f(1)=1$ и $f(\infty)=\infty.$ (Отображение
$f:D\rightarrow {\Bbb R}^n$ будем называется $Q$-квазиконформным,
если $f$ -- гомеоморфизм класса $W_{loc}^{1, n},$ для которого
функция $K_I(x, f)\le Q$ при почти всех $x\in D$). Пусть $h\in
\frak{F}_Q$ имеет комплексную характеристику вида
$\mu(z)=k(|z|)\frac{z}{\overline{z}},$ где $k(\tau):{\Bbb
R}\rightarrow {\Bbb B}^2$ -- измеримая функция. Тогда
\begin{equation}\label{eq1.9.1}
h(z)=\frac{z}{\overline{z}}\exp\left\{\int\limits_{1}^{|z|}\frac{1+k(\tau)}
{1-k(\tau)}\frac{d\tau}{\tau}\right\}\,,\end{equation}
см. \cite[предложение~6.5, гл.~6]{GR$_1$}. Полагаем $k(\tau)=\tau$
при $|\tau|<1$ и $k(\tau)=0$ при $|\tau|\ge 1,$ и комплексную
характеристику $\mu,$ определённую по правилу
$\mu(z)=k(|z|)\frac{z}{\overline{z}}.$ Ввиду сказанного выше,
отображение $h(z),$ заданное соотношением (\ref{eq1.9.1}) при
выбранной функции $k,$ является квазиконформным. Рассмотрим точку
$z_0=0.$ Посредством непосредственного подсчёта убеждаемся, что
$D_f(z, 0)=\frac{1-|z|}{1+|z|}$ при $z\in {\Bbb B}^2$ и $D_f(z,
0)=1$ при $z\not\in {\Bbb B}^2.$ Таким образом, отображение
$\widetilde{h}:=h|_{{\Bbb B}^2}$ имеет угловую дилатацию $D_f(z,
0),$ всюду меньшую единицы в единичном круге.

Нетрудно привести пример ограниченного отображения
$f:D\setminus\{b\}\rightarrow {\Bbb R}^n$ с конечным искажением
длины, для которого соответствующая угловая дилатация $Q:=D_f(x, b)$
удовлетворяет условию $Q\in FMO(b),$ равно как и условиям
$\int\limits_{0}^{\varepsilon_0}\frac{dt}{tq_{b}^{1/(n-1)}(t)}=\infty,$
$\int\limits_{\varepsilon}^{\varepsilon_0}\frac{dt}{tq_{b}^{1/(n-1)}(t)}<\infty$
при малых $\varepsilon>0,$ однако, этим же условиям, в то же время,
не удовлетворяет функция $\widetilde{Q}=K_I(x, f),$ где $K_I(x, f)$
определена в (\ref{equa16}). (Прямыми вычислениями можно показать,
что при $n=2$ всегда $K_I(z, f)=\frac{1+|\mu(z)|}{1-|\mu(z)|},$ где
$\mu(z)=\mu_f(z)=f_{\overline{z}}/f_z,$ когда $f_z \ne 0,$ и
$\mu(z)=0$ в противном случае). Для этой цели снова рассмотрим
случай $n=2$ и $D:={\Bbb B}^2\setminus\{0\}.$ Положим $b=0,$
$\mu(z)=k(|z|)\frac{z}{\overline{z}},$ тогда соответствующая угловая
дилатация $D_f(z, 0)$ равна: $D_f(z, 0)=\frac{1-|z|}{1+|z|}.$
Полагая $Q(z):=D_f(z, 0)$ и вычисляя $q_0(z)$ по правилу
(\ref{eq17}), заметим, что,
$\int\limits_{0}^{\varepsilon_0}\frac{dt}{tq_{0}(t)}=\infty$ и
$\int\limits_{\varepsilon}^{\varepsilon_0}\frac{dt}{tq_{0}(t)}<\infty$
при достаточно малых $\varepsilon>0.$ Кроме того, $D_f(z, 0)$ в этом
случае просто ограничена, что немедленно влечет условие $D_f(z,
0)\in FMO(0).$ Полагаем теперь
$k(r):=\frac{1+r^{\beta}}{1-r^{\beta},}$ где $\beta$ произвольное
фиксированное число из интервала $(0, 1).$ Поскольку
$\mu(z)=k(|z|)\frac{z}{\overline{z}},$ то  $K_I(z,
f)=\frac{1+|\mu(z)|}{1-|\mu(z)|}=\frac{1+k(|z|)|}{1-k(|z|)|}=\frac{1}{|z|^{\beta}}.$
Полагая $\widetilde{Q}(z):=K_I(z, f),$ мы видим, что при каждом
$\varepsilon_0>0, \varepsilon_0<1,$
$\int\limits_{0}^{\varepsilon_0}\frac{dr}{r\widetilde{q}_0(r)}=
\int\limits_{0}^{\varepsilon_0}\frac{dr}{r^{1-\beta}}<\infty.$
Используя предложение \ref{pr4*!}, покажем также, что
$\widetilde{Q}(z)\not\in FMO(0.)$ Действительно, прямые вычисления
показывают, что при достаточно малых $0<\varepsilon<\varepsilon_0$ и
некоторой постоянной $C>0$ выполнено $\int\limits_{\varepsilon<|z|<
{\varepsilon_0}}\frac{\widetilde{Q}(z)\, dm(z)} {\left(|z| \log
\frac{1}{|z|}\right)^2}=\omega_{n-1}\int\limits_{\varepsilon}^{\varepsilon_0}\frac{dr}
{r^{1+\beta}\log^2\frac{1}{r}}\ge \frac{C}{\varepsilon^{\beta/2}},$
откуда вытекает, что соотношение (\ref{equa17}) не выполнено;
следовательно, $\widetilde{Q}(z)\not\in FMO(0).$

Осталось показать, что найдётся открытое дискретное отображение $f$
с конечным искажением длины в ${\Bbb B}^2\setminus\{0\},$ которому
соответствуют построенные функции $\widetilde{Q}(z):=K_I(z, f)$ и
$Q(z):=D_f(z, 0).$ В самом деле, в силу
\cite[предложение~6.4]{RSY$_2$} найдётся гомеоморфизм $f:{\Bbb
B}^2\setminus\{0\}\rightarrow {\Bbb C}$ вида
$f(z)=e^{i\theta+\frac{1}{\beta}\left(\frac{r^{\beta}-1}{r^{\beta}}\right)},$
$z=re^{i\theta},$ которому соответствует выбранная функция
$\mu(z)=k(|z|)\frac{z}{\overline{z}},$ а также функции
$\widetilde{Q}(z):=K_I(z, f)$ и $Q(z):=D_f(z, 0),$ указанные выше.
Заметим, что $f$ -- гомеоморфизм класса $W_{loc}^{1, 2}({\Bbb
B}^2\setminus\{0\}),$ при этом, т.к. $K_I(z,
f)=\frac{1}{|z|^{\beta}}\in L_{loc}^{1}({\Bbb B}^2\setminus\{0\}),$
то $f^{\,-1}\in W_{loc}^{1, 2}(f({\Bbb B}^2\setminus\{0\}))$ (см.
\cite[следствие~2.3]{KO}). В таком случае, $f$ -- отображение с
конечным искажением длины (см. \cite[теоремы~8.1 и 8.6]{MRSY}, см.,
также, \cite[теоремы~4.6 и 6.10]{MRSY$_1$}), что и следовало
установить.

КОНТАКТНАЯ ИНФОРМАЦИЯ

\medskip
\noindent{{\bf Салимов Руслан Радикович, \\ Севостьянов Евгений
Александрович}
\\Институт
прикладной математики и механики НАН Украины \\
83 114 Украина, г. Донецк, ул. Розы Люксембург, д. 74, \\отдел
теории функций, раб. тел. (380) -- 62 -- 311 01 45, \\ 
e-mail: brusin2006@rambler.ru; ruslan623@yandex.ru}

\end{document}